\documentclass{article}                
\usepackage{graphicx}
\usepackage{mathptmx}      
\usepackage{lineno,hyperref}
\modulolinenumbers[5]
\usepackage{amsmath}
\usepackage{amstext}
\usepackage{amsfonts}
\usepackage{amssymb}
\usepackage{amsthm}
\usepackage{mathtools}
\usepackage{authblk}
\usepackage{natbib}
\usepackage{titlesec}
\titleformat*{\section}{\large\bfseries}
\newtheoremstyle{MyTheorem}
{\topsep}       
{\topsep}       
{\itshape}  
{}          
{\bfseries}  
{.}         
{5pt plus 1pt minus 1pt}      
{}          
\theoremstyle{MyTheorem}
\newtheorem{theorem}{Theorem}
\newtheorem{cond}[theorem]{Condition}
\newtheorem{observ}[theorem]{Observation}
\newtheorem{remark}[theorem]{Remark}
\newcommand\IPst[0]{\text{s.t.\;\,}}
\newcommand\IPmin[0]{\min\;}

\newcommand\cref[1]{(\ref{#1})}
\usepackage{etoolbox}
\apptocmd\normalsize{%
\setlength{\abovedisplayskip}{6pt plus 3pt minus 4pt}
\setlength{\belowdisplayskip}{6pt plus 3pt minus 4pt}
}{}{\undefined}
\begin{document}
\title{Compact Linearization for Binary Quadratic Programs Comprising Linear Constraints}
\author{Sven Mallach}
\affil{Institut f\"ur Informatik \authorcr Universit\"at zu K\"oln, 50923 K\"oln, Germany}
\date{August 23, 2018}
\maketitle
\begin{abstract}
In this paper, the compact linearization approach originally
proposed for binary quadratic programs with assignment constraints
is generalized to such programs with arbitrary linear equations and
inequalities that have positive coefficients and right hand sides.
Quadratic constraints may exist in addition, and the technique may
as well be applied if these impose the only nonlinearities, i.e., the
objective function is linear. We present special cases of linear
constraints (along with prominent combinatorial optimization problems
where these occur) such that the associated compact
linearization yields a linear programming relaxation that is provably
as least as strong as the one obtained with a classical linearization
method. Moreover, we show how to compute a compact linearization
automatically which might be used, e.g., by general-purpose
mixed-integer programming solvers.
\end{abstract}
\section{Introduction}\label{s:Intro}
In this paper, we consider a class of binary quadratic programs (BQPs)
that comprise a set of linear equations with positive coefficients
and right hand sides. More formally, we study mixed-integer non-linear
programs, that can (after applying some method of linearization to
realize the identities~(\ref{bqp:bilinear})) be
stated in the following general form that covers several NP-hard
combinatorial optimization problems including, e.g., the quadratic
assignment problem and the quadratic traveling salesman problem:
\begin{align}
\IPmin  &  c^Tx + d^Ty                       &      &            && \nonumber \\
\IPst   &  C x + D y                         &\ge\; &  e         && \label{bqp:linearized} \\
        &  \sum_{i \in A_k} \alpha_i^k x_{i} &=\;   &  \beta^k   && \mbox{for all } k \in K_E   \label{bqp:eqn} \\
        &  \sum_{i \in A_k} \alpha_i^k x_{i} &\le\; &  \beta^k   && \mbox{for all } k \in K_I   \label{bqp:ineq} \\
	&  y_{ij}                            &=\;   & x_i x_j    && \mbox{for all } (i, j) \in P \label{bqp:bilinear} \\
	&  x_i                               &\in\; & \{0,1\}    && \mbox{for all } i \in N   \nonumber
\end{align}
Here, the set of original binary variables $x$ in the BQP is indexed by
a set $N = \{ 1, \dots, n\}$ where $n \in \mathbb{N}^{>0}$.
Without loss of generality, we assume all bilinear terms $x_i x_j$
to be collected in an ordered set ${P \subset N \times N}$ such
that $i \le j$ for each ${(i,j) \in P}$. These are permitted to
occur in the objective function as well as in the set of
constraints, i.e., there may be an arbitrary set of $m \ge 0$
equations or inequalities that can be brought into the
form~(\ref{bqp:linearized}) after linearization and
where $C \in \mathbb{R}^{m \times n}$ and ${D \in \mathbb{R}^{m \times |P|}}$.
Apart from these, the BQP shall comprise a nonempty collection $K \coloneqq K_E \cup K_I$
of linear equations (\ref{bqp:eqn}) and linear inequalities (\ref{bqp:ineq}), where, for
each $k \in K$, $A_k$ is an index set specifying the binary variables on the
left hand side, $\alpha_i^k \in \mathbb{R}^{>0}$
for all $i \in A_k$, and ${\beta^k \in \mathbb{R}^{>0}}$.
The central subject of this paper
is a specialized and compact technique to implement the relations~(\ref{bqp:bilinear})
for this particular type of a BQP. In this context, \lq compact\rq\ means that the
approach typically involves the addition of significantly less constraints to the
problem formulation compared to the well-known and widely applied
linearization method by \citet{GloverWoolsey74} where relations~(\ref{bqp:bilinear})
are implemented using a variable $y_{ij} \in [0,1]$ and three constraints:
\begin{align}
  y_{ij}  &\le\;  x_{i} && \label{pck:stdlin1} \\
  y_{ij}  &\le\;  x_{j} && \label{pck:stdlin2} \\
  y_{ij}  &\ge\;  x_{i} + x_{j} - 1 && \label{pck:stdlin3}
\end{align}
The only prerequisite to apply the technique being the subject
of this paper is that, for each product $x_i x_j$, there need
to exist indices $k, \ell \in K$ such that $i \in A_k$ and ${j \in A_\ell}$, i.e., for
each variable being part of a bilinear term, there must be some linear
equation~(\ref{bqp:eqn}) or inequality~(\ref{bqp:ineq}) involving it.
It is however soft in the sense that single products not adhering to this
requirement do not affect a consistent linearization of those that do, and
these could still be linearized using, e.g., the method just described. 
Moreover, if the prerequisite is satisfied already by considering only
the present equations $K_E$, inequalities $K_I$, or a subset of them,
then it might be attractive to apply the technique only to these in order
to obtain a strong continuous relaxation.
Originally, the technique was proposed
by \citet{Liberti2007} and revised by \citet{Mallach2017}
for the case where the only equations (\ref{bqp:eqn}) considered
are assignment (or rather \lq single selection\rq)
constraints, i.e., where for all $k \in K_E$ one has
$\beta^k = 1$ and $a^k_i = 1$ for all~$i \in A_k$.
In this paper, we show that the underlying methodology of the compact
linearization technique can be generalized to BQPs of the form displayed
above. We then~investigate under which circumstances the
linear programming relaxation of the obtained formulation is provably
as least as strong as when using the linearization by \citet{GloverWoolsey74}.
Moreover, we highlight some prominent combinatorial optimization
problems where previously found mixed-integer programming formulations appear
now as particular compact linearizations. Last but not
least, we show how these can be computed automatically,
e.g., as part of a general-purpose mixed-integer programming solver.
\section{Related Linearization Methods for BQPs}
Since linearizations of quadratic and, more generally, polynomial programming
problems, enable the application of well-studied mixed-integer linear programming
techniques, they have been an active field of research since the 1960s.
The seminal idea to model binary conjunctions using additional (binary) variables
and inequalities~(\ref{pck:stdlin3}) combined with the inequalities
$x_i + x_j - 2y_{ij} \ge 0$ is attributed to~\citet{Fortet59,Fortet60}, and discussed
in several succeeding books and papers, e.g.~by \citet{Balas64}, \citet{Zangwill}, \citet{Watters},
\citet{HammerRudeanu}, and by \citet{GloverWoolsey73}.
Shortly thereafter, \citet{GloverWoolsey74} found
that an explicit integrality requirement on $y_{ij}$ becomes obsolete when replacing
the mentioned inequality with~(\ref{pck:stdlin1}) and~(\ref{pck:stdlin2}). This
method, that is henceforth referred to as the \lq Glover-Woolsey linearization\rq, is
until today regarded as \lq the standard linearization technique\rq.
Together with $y_{ij} \ge 0$, the Glover-Woolsey linearization appears as a special
case of the convex envelopes for general nonlinear programming
problems as proposed by~\cite{McCormick1976}.
Moreover, \cite{Padberg1989} proved the corresponding four inequalities to be
facet-defining for the polytope associated to unconstrained binary quadratic
optimization problems:
$$QP^n = conv \{ (x,y) \in \mathbb{R}^n \times \mathbb{R}^{n(n-1)/2} \mid x \in \{0,1\}^n,\; y_{ij} = x_i x_j \mbox{ for all } 1 \le i < j \le n \}$$
While the Glover-Woolsey linearization is always applicable, its inequalities do
not couple related linearization variables (i.e., such sharing a common factor)
if these are present. Depending on the concrete problem to be solved, this may
result in rather weak linear programming relaxations.
Refined techniques 
are however (almost) exclusively available for BQPs with no or only linear
constraints. This is true, e.g., for the posiform-based techniques by \citet{HansenMeyer2009},
the \lq Clique-Edge Linearization\rq\ by \citet{GueyeMichelon2009}, the
\lq Extended Linear Formulation\rq\ by \citet{FuriniTraversi}, as well as 
earlier ones by \citet{Glover75}, \citet{OralKettani92a,OralKettani92b},
\citet{CPP2004}, and \citet{SheraliSmith2007}.
An exception is the well-known transformation between unconstrained
binary quadratic optimization and the maximum cut problem (cf.~\cite{Hammer1965,DeSimone90})
that, in principle, also allows for the translation of a (possibly quadratic)
constraint set.
The only other general methodology to convert quadratically constrained BQPs
into mixed-integer linear programs is, to the best of our knowledge, the
reformulation-linearization technique (RLT) by~\citet{SheraliAdamsRLTBook}.
The compact linearization technique
can be interpreted as a particular and usually incomplete
(or \lq sparse\rq) first level application of the RLT, and
establishes a practical and general approach to linearize
an important and rich subclass of BQPs.
\section{Compact Linearization}\label{s:CmpLin}
The compact linearization approach for binary quadratic
problems with linear constraints is as follows.
With each linear equation of type (\ref{bqp:eqn}),
i.e., with its index set $A_k$, we associate a
corresponding index set $B^E_k \subseteq N$, and with
each linear inequality of type (\ref{bqp:ineq}),
we associate two such index sets
$B^{I_+}_k \subseteq N$ and $B^{I_-}_k \subseteq N$.
For ease of subsequent reference, let
$B_k \coloneqq B^E_k$ if $k \in K_E$,
$B_k \coloneqq B^{I_+}_k \cup B^{I_-}_k$ if $k \in K_I$,
and $B \coloneqq \bigcup_{k \in K} B_k$.
For~each $j \in B^E_k \cup B^{I_+}_k$, we then
multiply the respective equation or inequality by $x_j$,
and for each $j \in B^{I_-}_k$, we multiply the respective
inequality by $1 - x_j$.
We thus obtain the new~constraints:
\begin{align}
    \sum_{i \in A_k} \alpha_i^k x_i x_j     &=\;  \beta^k x_{j} && \mbox{for all } j \in B^E_k,\; \mbox{for all } k \in K_E \label{cmp:eqn0}\\
    \sum_{i \in A_k} \alpha_i^k x_i x_j     &\le\;  \beta^k x_{j} && \mbox{for all } j \in B^{I_+}_k,\; \mbox{for all } k \in K_I \label{cmp:ineq+0}\\
    \sum_{i \in A_k} \alpha_i^k x_i (1 - x_j) &\le\;  \beta^k (1 - x_{j}) && \mbox{for all } j \in B^{I_-}_k,\; \mbox{for all } k \in K_I \label{cmp:ineq-0}
\end{align}
Each product $x_i x_j$ induced by any of these new
equations or inequalities is then
replaced by a continuous linearization variable $y_{ij}$
(if $i \le j$) or $y_{ji}$ (otherwise). We denote the set
of bilinear terms created this way with
$$Q = \{ (i,j) \mid i \le j \mbox { and } \exists k \in K: i \in A_k \mbox { and } j \in B_k \mbox{, or } j \in A_k
\mbox { and } i \in B_k \}.$$
Rewriting \cref{cmp:eqn0}--\cref{cmp:ineq-0} using $Q$, we obtain the linearization constraints:
\begin{align}
    \sum_{\mathclap{i \in A_k, (i,j) \in Q}} \alpha_i^k y_{ij} \;\; + \;\; \sum_{\mathclap{i \in A_k, (j,i) \in Q}} \alpha_i^k y_{ji}    &=\; \beta^k x_{j} && \mbox{for all } j \in B^E_k,\; \mbox{for all } k \in K_E \label{cmp:eqn1} \\
    \sum_{\mathclap{i \in A_k, (i,j) \in Q}} \alpha_i^k y_{ij} \;\; + \;\; \sum_{\mathclap{i \in A_k, (j,i) \in Q}} \alpha_i^k y_{ji}    &\le\; \beta^k x_{j} && \mbox{for all } j \in B^{I_+}_k,\; \mbox{for all } k \in K_I \label{cmp:ineq+1} \\
    \sum_{\mathclap{\;\;\;\; i \in A_k, (i,j) \in Q}} \alpha_i^k (x_i - y_{ij})  + \; \sum_{\mathclap{\;\;\;\;i \in A_k, (j,i) \in Q}} \alpha_i^k (x_i - y_{ji})  &\le\; \beta^k (1 - x_{j}) \!\!\!\! && \mbox{for all } j \in B^{I_-}_k,\; \mbox{for all } k \in K_I \label{cmp:ineq-1}
\end{align}
It is clear that the constraints (\ref{cmp:eqn0})--(\ref{cmp:ineq-0})
are valid for the original problem and so are thus as well the
constraints (\ref{cmp:eqn1})--(\ref{cmp:ineq-1}) whenever the
introduced linearization variables take on consistent values with
respect to their two original counterparts, i.e.,
$y_{ij} = x_i x_j$ holds for all $(i, j) \in Q$.
Since the original problem formulation comprises the bilinear
terms defined by the set $P$,
we need to choose the set $B$ such that the induced set of
variables $Q$ will be equal to $P$ or contain $P$ as a subset. We
will discuss how to determine such a set $Q \supseteq P$ in
Sect.~\ref{s:Compute}, but suppose for now that it is already
at hand. We will show that a consistent linearization is obtained
if and only if the following three conditions are satisfied:
\begin{cond} \label{cond:c1}
For each $(i,j) \in Q$, there is a $k \in K$
such that $i \in A_k$ and $j \in B^E_k \cup B^{I_+}_k$.
\end{cond}
\begin{cond} \label{cond:c2}
For each $(i,j) \in Q$, there is an $\ell \in K$
such that $j \in A_\ell$ and $i \in B^E_\ell \cup B^{I_+}_\ell$.
\end{cond}
\begin{cond} \label{cond:c3}
For each $(i,j) \in Q$, there is a $k \in K$
such that $i \in A_k$ and $j \in B^E_k \cup B^{I_-}_k$
\emph{or} an $\ell \in K$ such that $j \in A_\ell$
and $i \in B^E_\ell \cup B^{I_-}_\ell$.
\end{cond}
Importantly, $k = \ell$ is a valid choice for
satisfying Conditions~\ref{cond:c1} and
\ref{cond:c2}, and Condition~\ref{cond:c3} is
implicitly satisfied whenever
Condition \ref{cond:c1} 
\emph{or} Condition
\ref{cond:c2} is established 
using an equation.
In particular, Condition~\ref{cond:c3} is obsolete
if only linear equations~(\ref{bqp:eqn}) but no
inequalities~(\ref{bqp:ineq}) are present
in the program to be linearized.
\begin{theorem} \label{thm:main}
For any integer solution $x \in \{0,1\}^n$, the
linearization constraints~(\ref{cmp:eqn1})--(\ref{cmp:ineq-1})
imply $y_{ij} = x_i  x_j$ for all $(i,j) \in Q$ if and only if
Conditions~\ref{cond:c1}--\ref{cond:c3} are satisfied.
\end{theorem}
\begin{proof}
Let $(i,j) \in Q$.
By Condition~\ref{cond:c1}, there is a $k \in K$ such that
$i \in A_k$, $j \in B^E_k \cup B^{I_+}_k$ and hence either the equation
\begin{align}
    \sum_{h \in A_k, (h,j) \in Q} \alpha_h^k y_{hj} + \sum_{h \in A_k, (j,h) \in Q} \alpha_h^k y_{jh}   &=\; \beta^k x_{j} && \tag{$*_{Ej}$} \label{eqn:rhsj_proof_Ej}
\end{align}
or the inequality
\begin{align}
    \sum_{h \in A_k, (h,j) \in Q} \alpha_h^k y_{hj} + \sum_{h \in A_k, (j,h) \in Q} \alpha_h^k y_{jh}   &\le\; \beta^k x_{j} && \tag{$*_{Ij+}$} \label{eqn:rhsj_proof_Ij+}
\end{align}
exists and has $y_{ij}$ on its left hand side. Since
$\alpha_h^k > 0$ for all $h \in A_k$ and $0 \le y_{ij} \le 1$,
each of them establishes that $y_{ij} = 0$ whenever $x_j = 0$.
Similarly, by Condition~\ref{cond:c2}, there is an $\ell \in K$ such that
$j \in A_\ell$, $i \in B^E_\ell \cup B^{I_+}_\ell$ and hence the equation
\begin{align*}
    \sum_{h \in A_\ell, (h,i) \in Q} \alpha_h^\ell y_{hi} + \sum_{h \in A_\ell, (i,h) \in Q} \alpha_h^\ell y_{ih}   &=\; \beta^\ell x_{i} &&  \tag{$*_{Ei}$} \label{eqn:rhsi_proof_Ei}
\end{align*}
or the inequality
\begin{align*}
    \sum_{h \in A_\ell, (h,i) \in Q} \alpha_h^\ell y_{hi} + \sum_{h \in A_\ell, (i,h) \in Q} \alpha_h^\ell y_{ih}   &\le\; \beta^\ell x_{i} &&  \tag{$*_{Ii+}$} \label{eqn:rhsi_proof_Ii+}
\end{align*}
exists and has $y_{ij}$ on its left hand side. Since
$\alpha_h^\ell > 0$ for all $h \in A_\ell$ and $0 \le y_{ij} \le 1$,
each of them establishes that $y_{ij} = 0$ whenever $x_i = 0$.
Let now $x_i = x_j = 1$. By Condition~\ref{cond:c3}, we either have
at least one equation or at least one inequality relating $y_{ij}$ to
either $x_i$ or $x_j$. Consider first the equation case, and suppose
w.l.o.g.~that equation (\ref{eqn:rhsj_proof_Ej}) exists (the
opposite case with (\ref{eqn:rhsi_proof_Ei}) can be exploited analogously).
If $y_{ij} = 1$, there is nothing to show, so suppose that $y_{ij} < 1$ which
means that we are in the following situation:
\begin{align}
    \sum_{h \in A_k, (h,j) \in Q, h \neq i} \!\!\!\! \alpha_h^k y_{hj} + \sum_{h \in A_k, (j,h) \in Q, h \neq i} \!\!\!\! \alpha_h^k y_{jh}  &=\; \beta^k \underbrace{x_j}_{=1} -\; \alpha_i^k \underbrace{y_{ij}}_{< 1} > \beta^k - \alpha_i^k && \tag{$*'_{Ej}$} \label{eqn:rhsj_proof2_Ej}
\end{align}
At the same time, we also have $\sum_{h \in A_k, h \neq i} \alpha_h^k x_h = \beta^k - \alpha_i^k$ with $x_h \in \{0,1\}$.
In order for the equation~(\ref{eqn:rhsj_proof_Ej}) to be satisfied, an additional amount
of $(1 -y_{ij})\alpha_i^k > 0$ thus needs to be contributed by the other summands on
the left hand side of~(\ref{eqn:rhsj_proof2_Ej}).
This implies, however, that there must be some $h \in A_k$, $h \neq i$, such that $y_{hj} > 0$
(or $y_{jh} > 0$) while $x_h = 0$ -- which is impossible since
Conditions~\ref{cond:c1} and~\ref{cond:c2} are established for these variables as well.
Finally, we consider the inequality case and assume again w.l.o.g.~that
Condition~\ref{cond:c3} is satisfied by some $k \in K_I$ with
$i \in A_k$ and $j \in B^{I_-}_k$, i.e., such that the the inequality
\begin{align}
    \sum_{h \in A_k, (h,j) \in Q} \alpha_h^k (x_h - y_{hj}) + \sum_{h \in A_k, (j,h) \in Q} \alpha_h^k (x_h - y_{jh})   &\le\; \beta^k (1 - x_{j}) && \tag{$*_{Ij-}$} \label{eqn:rhsj_proof_I-}
\end{align}
exists. Its right hand side now evaluates to zero since $x_j = 1$.
Looking at the left hand side, for any $h \in A_k$
(including $i$) the terms $(x_h - y_{hj})$ respectively $(x_h - y_{jh})$
cannot be negative since $y_{hj}$ ($y_{jh}$) must be zero if $x_h$ is
(by the arguments above) and cannot be larger than one if $x_h$ is
(by its upper bound). Moreover, since the right hand side is zero
and $\alpha_h^k > 0$ for all $h \in A_k$, the terms
cannot be positive as well. It follows that $x_h = y_{hj}$ for
all $h \in A_k$ (including i) and thus $y_{ij} = 1$ as desired.
We have just shown the \emph{sufficiency} of the equations
induced by satisfying Conditions~\ref{cond:c1}--\ref{cond:c3}.
Moreover, within a framework that constructs a linearization only
by means of constraints of type (\ref{cmp:eqn1})--(\ref{cmp:ineq-1}),
it is impossible to enforce $y_{ij} = 0$ if $x_i = 0$
other than by satisfying Condition~\ref{cond:c1}, impossible to
enforce $y_{ij} = 0$ if $x_j = 0$ other than by satisfying
Condition~\ref{cond:c2}, and no other way to ensure $y_{ij} = 1$
if both $x_i$ and $x_j$ are equal to one as well than by satisfying
Condition~\ref{cond:c3}, which implies their \emph{necessity}.
\end{proof}
Theorem~\ref{thm:main} establishes that
Conditions~\ref{cond:c1}--\ref{cond:c3} are the
only relevant criteria for that
inequalities~(\ref{pck:stdlin1})--(\ref{pck:stdlin3})
be implied for integer solutions $x \in \{0,1\}^n$
for a particular $(i,j) \in Q$ -- allowing for the
construction of \lq compact\rq\ linearizations
of a given demanded \lq sparse\rq\ set of products
$P \subseteq Q$ based on an arbitrary given
linear constraint set. Known before from the
RLT (cf.~\citet{SheraliAdamsRLTZeroOne}) has
been the fact that inequalities~(\ref{pck:stdlin1})--(\ref{pck:stdlin3})
are implied for a \emph{complete} $P$, i.e.,
$P = \{(i,j) \mid i,j \in N, i < j\}$ if
a set of constraints comprising in total \emph{all}
$x_i$, $i \in N$, is multiplied by \emph{all}
these variables and, in case of inequalities, by their
complements ($1 - x_i$), which \emph{obviously} satisfies
Conditions~\ref{cond:c1}--\ref{cond:c3}.
\vspace{-0.1cm}
\section{LP Relaxation Strength of Compact Linearizations}
While this is unfortunately not possible for the
general case, we can prove that a compact linearization
yields a linear programming relaxation that is as least
as tight as the one obtained with the Glover-Woolsey linearization
if the structure of the present linear constraints is more specific.
In particular, the next two subsections together show that this
is the case if Conditions~\ref{cond:c1}--\ref{cond:c3} are satisfied
based on a selection of \lq assignment\rq\ and \lq knapsack\rq\ constraints.
\subsection{Compact Linearizations with Provably Strong LP Relaxations}\label{ss:ProveStrong}
\subsubsection{Assignment or \lq Single Selection\rq\ Equations}
Let us first consider the case where the
equations (\ref{bqp:eqn}) are assignment
(or rather \lq single selection\rq) constraints, i.e., $K = K_E$,
$a^k_i = 1$ for all $i \in A_k$ and $\beta^k = 1$ for all $k \in K$.
This was the application the compact linearization technique
was originally proposed for by~\citet{Liberti2007}.
Later, \citet{Mallach2017} clarified that a consistent
linearization is obtained if and only if
Conditions~\ref{cond:c1} and~\ref{cond:c2} are enforced.
Accidentally, and in contrast to inequalities~(\ref{pck:stdlin1}) and (\ref{pck:stdlin2}),
the proof did not verify that inequalities (\ref{pck:stdlin3}) hold as well in case of
fractional solutions $x \in [0,1]^n$. This is caught up on now by giving
a complete proof of the following theorem.
\begin{theorem} \label{thm:ass}
Let $\beta^k = 1$ for all $k \in K$, and as well $a^k_i = 1$ for each $i \in A_k$, $k \in K$.
Then for any solution $x \in [0,1]^n$, the inequalities $y_{ij} \le x_i$, $y_{ij} \le x_j$ and
$y_{ij} \ge x_i + x_j - 1$ are implied by equations~(\ref{cmp:eqn1}) for all $(i,j) \in Q$
if and only if Conditions~\ref{cond:c1} and~\ref{cond:c2} are~satisfied.
\end{theorem}
\begin{proof}
Let $(i,j) \in Q$.
By Condition~\ref{cond:c1}, there is a $k \in K_E$ such that
$i \in A_k$, $j \in B^E_k$ and hence the equation
\begin{align}
    \sum_{h \in A_k, (h,j) \in Q} y_{hj} + \sum_{h \in A_k, (j,h) \in Q} y_{jh}   &=\; x_{j} &&  \label{eqn:rhsj_proof3}
\end{align}
exists, has $y_{ij}$ on its left hand side, and thus establishes $y_{ij} \le x_j$.
Similarly, by Condition~\ref{cond:c2}, there is an $\ell \in K_E$ such that
$j \in A_\ell$, $i \in B^E_\ell$ and hence the equation
\begin{align}
    \sum_{h \in A_\ell, (h,i) \in Q} y_{hi} + \sum_{h \in A_\ell, (i,h) \in Q} y_{ih}   &=\; x_{i} &&  \label{eqn:rhsi_proof3}
\end{align}
exists, has $y_{ij}$ on its left hand side, and thus establishes $y_{ij} \le x_i$.
To show that $y_{ij} \ge x_i + x_j - 1$, consider equation~(\ref{eqn:rhsj_proof3}) in
combination with its original counterpart $\sum_{h \in A_k} x_h = 1$.
For any $y_{hj}$ (or $y_{jh}$) in~(\ref{eqn:rhsj_proof3}), the Conditions~\ref{cond:c1}
and~\ref{cond:c2} assure that there is an equation establishing $y_{hj} \le x_h$ ($y_{jh} \le x_h$).
Thus we have
\begin{align}
    \sum_{h \in A_k, (h,j) \in Q, h \neq i} y_{hj} + \sum_{h \in A_k, (j,h) \in Q, h \neq i} y_{jh}   &\le\; \sum_{h \in A_k, h \neq i} x_{h} = 1 - x_i && \nonumber 
\end{align}
Applying this upper bound within equation~(\ref{eqn:rhsj_proof3}), we obtain:
$$y_{ij} + \underbrace{\sum_{h \in A_k, (h,j) \in Q, h \neq i} y_{hj} + \sum_{h \in A_k, (j,h) \in Q, h \neq i} y_{jh}}_{\le 1 - x_i}  = x_j\; \Leftrightarrow y_{ij} \ge x_i + x_j - 1$$
Finally, the necessity to satisfy Conditions~\ref{cond:c1} and~\ref{cond:c2}
is given for the same reasons as mentioned after the proof of Theorem~\ref{thm:main}.
\end{proof}
Again, special cases of Theorem~\ref{thm:ass} where \emph{all}
present assignment constraints are multiplied by \emph{all}
variables $x_i$, $i \in N$ and where $P$ contains
\emph{all} possible products of these, were shown before for
the quadratic assignment (\citet{AdamsJohnsonQAP}) and
quadratic semi-assignment (\citet{Elloumi2001}) problems.
\subsubsection{\lq Knapsack\rq\ Inequalities}
We now focus on the case where the only
constraints taken into account are \lq knapsack\rq\
constraints, i.e., $K = K_I$,
$a^k_i = 1$ for all $i \in A_k$ and $\beta^k = 1$ for all $k \in K$.
\begin{theorem} \label{thm:knapsack}
Let $\beta^k = 1$ for all $k \in K$, and as well $a^k_i = 1$ for each $i \in A_k$, $k \in K$.
Then for any solution $x \in [0,1]^n$, the inequalities $y_{ij} \le x_i$, $y_{ij} \le x_j$ and
$y_{ij} \ge x_i + x_j - 1$ are implied by inequalities~(\ref{cmp:ineq+1}) and (\ref{cmp:ineq-1})
for all $(i,j) \in Q$ if and only if the Conditions~\ref{cond:c1}--\ref{cond:c3} are satisfied.
\end{theorem}
\begin{proof}
Let $(i,j) \in Q$.
By Condition~\ref{cond:c1}, there is a $k \in K$ such that
$i \in A_k$, $j \in B^{I_+}_k$ and hence the inequality
\begin{align}
    \sum_{h \in A_k, (h,j) \in Q} y_{hj} + \sum_{h \in A_k, (j,h) \in Q} y_{jh}   &\le\; x_{j} &&  \label{ineq:rhsj_proof3}
\end{align}
exists, has $y_{ij}$ on its left hand side, and thus establishes $y_{ij} \le x_j$.
Similarly, by Condition~\ref{cond:c2}, there is an $\ell \in K$ such that
    $j \in A_\ell$, $i \in B^{I_+}_\ell$ and hence the equation
\begin{align}
    \sum_{h \in A_\ell, (h,i) \in Q} y_{hi} + \sum_{h \in A_\ell, (i,h) \in Q} y_{ih}   &\le\; x_{i} &&  \label{ineq:rhsi_proof3}
\end{align}
exists, has $y_{ij}$ on its left hand side, and thus establishes $y_{ij} \le x_i$.
Moreover, by Condition~\ref{cond:c3}, there is, w.l.o.g., some $k \in K_I$ with
$i \in A_k$ and $j \in B^{I_-}_k$, i.e., such that the inequality
\begin{align}
    \sum_{h \in A_k, (h,j) \in Q} (x_h - y_{hj}) + \sum_{h \in A_k, (j,h) \in Q} (x_h - y_{jh})   &\le\; 1 - x_{j} && \label{ineq:rhsj_proof_I-}
\end{align}
exists. Due to Conditions~\ref{cond:c1} and~\ref{cond:c2}, we have that $x_h \ge y_{jh}$ for
each $h \in A_k, (j,h) \in Q$ and $x_h \ge y_{hj}$ for each
$h \in A_k, (h,j) \in Q$ in (\ref{ineq:rhsj_proof_I-}). By reordering the
latter to
\begin{align*}
    x_j + x_i - y_{ij} + \underbrace{\sum_{h \in A_k, (h,j) \in Q, h \neq i} (x_h - y_{hj}) + \sum_{h \in A_k, (j,h) \in Q, h \neq i} (x_h - y_{jh})}_{\ge 0} &\le\; 1, &&
\end{align*}
we obtain the desired result.
The necessity of Conditions~\ref{cond:c1}--\ref{cond:c3} stems once more from
the same reasons as mentioned in the proof of Theorem~\ref{thm:main}.
\end{proof}
\subsubsection{\lq Double Selection\rq\ Equations with Induced Squares}
Another important special case, where the equations
induced by Conditions~\ref{cond:c1} and~\ref{cond:c2}
imply the inequalities~(\ref{pck:stdlin1})--(\ref{pck:stdlin3})
also for fractional solutions, is obtained if the right
hand sides of all the considered original equations $K_E$ are
equal to two and all (or a subset of) the products to be induced are
exactly those given by $A_k \times A_k$ for all $k \in K_E$. In this
case, Conditions~\ref{cond:c1} and~\ref{cond:c2} are implicitly
satisfied for all these products by choosing $B_k = A_k$ for all
$k \in K_E$. We will see in Sect.~\ref{ss:QTSP} an application where
this case occurs in practice and that also gives an example where it
is attractive to apply the compact linearization only to a subset of
the present linear constraints.
\begin{theorem} \label{thm:tsp}
If, for all $k \in K_E$, (i) $a^k_i = 1$ for all $i \in A_k$, (ii) $\beta^k = 2$, and (iii) $B^E_k = A_k$, then there
is a compact linearization such that, for any solution $x \in [0,1]^n$, the inequalities
$y_{ij} \le x_i$, $y_{ij} \le x_j$ and
$y_{ij} \ge x_i + x_j - 1$ are implied by the equations~(\ref{cmp:eqn1}) for all $(i,j) \in Q$, $i \neq j$.
\end{theorem}
\begin{proof}
 Due to (i)--(iii), the induced equations~(\ref{cmp:eqn1}) look like:
\begin{align}
    y_{jj} + \sum_{h \in A_k, h < j} y_{hj}\ + \sum_{h \in A_k, j < h} y_{jh}    &=\; 2 x_{j} && \mbox{for all } j \in A_k,\; \mbox{for all } k \in K_E \nonumber 
\end{align}
Since $y_{jj}$ shall take on the same value as $x_j$, we may eliminate $y_{jj}$ on
the left and once subtract $x_j$ on the right. We obtain:
\begin{align}
   \sum_{h \in A_k, h < j} y_{hj}\ + \sum_{h \in A_k, j < h} y_{jh}    &=\; x_{j} && \mbox{for all } j \in A_k,\; \mbox{for all } k \in K_E \label{cmp:sp2_2}
\end{align}
These equations establish inequalities~(\ref{pck:stdlin1}) and (\ref{pck:stdlin2}) for all
$y_{ij}$, $(i,j) \in Q$, $i \neq j$.
Combining them with the original equations $\sum_{a \in A_k} x_a = 2$ yields the following identities:
\begin{align}
    2 = \sum_{a \in A_k} x_a = \sum_{a \in A_k} \big( \sum_{h \in A_k, h < a} y_{ha}\ + \sum_{h \in A_k, a < h} y_{ah} \big) = 2 * \sum_{a \in A_k} \sum_{h \in A_k, a < h} y_{ah} \nonumber 
\end{align}
As an immediate consequence, it follows (even for fractional $x$) that:
\begin{align}
    \sum_{a \in A_k} \sum_{h \in A_k, a < h} y_{ah} = 1 \label{cmp:sp2_3}
\end{align}
Since $\{i,j\} \subseteq A_k$, we obtain a subtotal of (\ref{cmp:sp2_3})
if we sum the equations~(\ref{cmp:sp2_2}) expressed for $i$ and for $j$
(which both contain $y_{ij}$ on their left hand sides). We can exploit
    this as follows (cf.~\cite{FischerDiss}) in order to show that $y_{ij} \ge x_i + x_j - 1$:
\begin{align*}
    x_i + x_j &=  \sum_{h \in A_k, i < h} y_{ih}\ + \sum_{h \in A_k, h < i} y_{hi} + \sum_{h \in A_k, j < h} y_{jh}\ + \sum_{h \in A_k, h < j} y_{hj}\\
              &=  y_{ij} + \sum_{h \in A_k, i < h \neq j} y_{ih}\ + \sum_{h \in A_k, j \neq h < i} y_{hi} + \sum_{h \in A_k, j < h} y_{jh}\ + \sum_{h \in A_k, h < j} y_{hj}\\
	      &\overset{(\ref{cmp:sp2_3})}{\le}\; y_{ij} + \sum_{a \in A_k} \sum_{h \in A_k, a < h} y_{ah} \\
              &= y_{ij} + 1
\end{align*}
\end{proof}
\begin{remark}
If $\beta^k > 2$ in Theorem~\ref{thm:tsp} or $\beta^k \ge 2$ in the general setting, then
it is impossible to conclude $y_{ij} \le x_i$ and $y_{ij} \le x_j$ from the
linearization equations for fractional $x$.
Moreover, if $B_k \neq A_k$, then it
is impossible to conclude $y_{ij} \ge x_i + x_j - 1$ from~(\ref{cmp:sp2_3}).
\end{remark}
\subsection{A Scenario with a Strictly Stronger LP Relaxation}\label{ss:Stronger}
Sect.~\ref{ss:ProveStrong} displayed scenarios where the proposed technique
provides a linear programming relaxation that is at least as strong as the one
obtained using the Glover-Woolsey linearization.
If the \emph{equation} sets $A_k$, $k \in K_E$, and $P$ allow to construct
a compact linearization with $Q = P$ (after possible squares are eliminated),
then it can be shown that the corresponding relaxation is even \emph{strictly}
stronger. For example, this is true for the applications presented in Sect.~\ref{s:Apps}.
Moreover, a set $Q$ generated can sometimes also be \emph{made}
compliant to this case or at least strengthened and at the same time reduced in size
by a postprocessing, if the particular problem at hand allows to fix the values of
(some of) the variables in $Q \setminus P$ prior to solving~it.
\begin{observ}\label{ob:MaxHalf}
Let $(i,j) \in P$ and suppose that $x_i > 0$, $x_j > 0$, and $x_i + x_j \le 1$
hold in a given optimum solution to the linear programming relaxation obtained
with~the Glover-Woolsey linearization. Then inequalities~(\ref{pck:stdlin3})
are dominated by the trivial ones, i.e., $y_{ij} \ge 0$.
Thus, if $d_{ij} > 0$ and
no other constraint enforces $y_{ij} > 0$, then~$y_{ij} = 0$.
\end{observ}
Now, for a fixed $i \in N$, let $k \in K_E$ be an index such that $i \in B^E_k$.
Under the assumptions made before, it is then readily seen that the corresponding
induced equation
$$\sum_{h \in A_k, h < i} \alpha_h^k y_{hi} + \sum_{h \in A_k, i < h} \alpha_h^k y_{ih}   =\; \beta^k x_{i}$$
of the associated compactly linearized program formulation cuts off any point
where the situation described in Observation~\ref{ob:MaxHalf} applies to \emph{all}
products $(i,h)$ or $(h,i) \in P$ with $h \in A_k$. Consequently, this is true
for \emph{all the points} whose components have non-zero entries for \emph{at least one}
$x_i$, $i\in N$, and all $x_h$, $h \in A_k$, associated to some $k \in K_E$ such that
$i \in B^E_k$, and zero-entries for all corresponding products $y_{ih}$ or $y_{hi}$.
It also becomes visible, why a set $Q=P$ needs to be assumed for this construction, as
otherwise the value $\beta^k x_{i}$ could be entirely assigned to some $y_{hi}$ or
$y_{ih}$ for $h \in A_k$ where $(h,i)$ respectively $(i,h)$ is in $Q \setminus P$.
Such a point however cannot occur using a Glover-Woolsey linearization as it only
linearizes the products in $P$. For a similar reason, it is also necessary to
establish consistency, i.e.~to satisfy Conditions~\ref{cond:c1}--\ref{cond:c3}, for
all products in $Q$ rather than just those in $P$.
\section{Obtaining a Compact Linearization (Automatically)}\label{s:Compute}
We shall now elaborate on how to obtain
a consistent linearization while inducing
a minimum number of additional constraints
and as well a set $Q \supseteq P$
as small as possible.
Such a \lq most compact\rq\ linearization
can be computed by solving the
following mixed-integer program:
\begin{align}
\SwapAboveDisplaySkip
    \IPmin  & \sum_{1 \le i \le n} \bigg( \sum_{k \in K_E} w_E z^E_{ik} + \sum_{k \in K_I} \Big( w_{I_+} z^{I_+}_{ik} + w_{I_-} z^{I_-}_{ik} \Big) \bigg) + w_{Q} \bigg( \sum_{1 \le i \le n}\sum_{i \le j \le n} f_{ij} \bigg)    \span \span  \span\span \nonumber \\
\IPst   &  f_{ij}                                         &=\;   & 1          && \mbox{for all } (i, j) \in P  \label{minB:EinF0}  \\
        &  f_{ij}                                         &\ge\; & z^E_{jk}     && \mbox{for all } k \in K_E, i \in A_k, j \in N, i \le j \label{minB:EinF1a} \\
	&  f_{ji}                                         &\ge\; & z^E_{jk}     && \mbox{for all } k \in K_E, i \in A_k, j \in N, j < i \label{minB:EinF1b} \\
        &  f_{ij}                                         &\ge\; & z^{I_+}_{jk}     && \mbox{for all } k \in K_I, i \in A_k, j \in N, i \le j \label{minB:EinF2a} \\
	&  f_{ji}                                         &\ge\; & z^{I_+}_{jk}     && \mbox{for all } k \in K_I, i \in A_k, j \in N, j < i \label{minB:EinF2b} \\
        &  f_{ij}                                         &\ge\; & z^{I_-}_{jk}     && \mbox{for all } k \in K_I, i \in A_k, j \in N, i \le j \label{minB:EinF3a} \\
	&  f_{ji}                                         &\ge\; & z^{I_-}_{jk}     && \mbox{for all } k \in K_I, i \in A_k, j \in N, j < i \label{minB:EinF3b} \\
	&  \sum_{\mathclap{k \in K_E: i \in A_k}} z^E_{jk}\;\; + \;\; \sum_{\mathclap{k \in K_I: i \in A_k}} z^{I_+}_{jk}   &\ge\; & f_{ij}       && \mbox{for all } 1 \le i \le j \le n \label{minB:EinF1} \\
	&  \sum_{\mathclap{k \in K_E: j \in A_k}} z^E_{ik}\;\; + \;\; \sum_{\mathclap{k \in K_I: j \in A_k}} z^{I_+}_{ik}   &\ge\; & f_{ij}       && \mbox{for all } 1 \le i \le j \le n \label{minB:EinF2} \\
	&  \sum_{\mathclap{k \in K_E: j \in A_k}} z^E_{ik}\;\; + \;\; \sum_{\mathclap{k \in K_I: j \in A_k} }z^{I_-}_{ik}\;\; + \;\; \nonumber \\
	&  \sum_{\mathclap{k \in K_E: i \in A_k}} z^E_{jk}\; + \;\; \sum_{\mathclap{k \in K_I: i \in A_k} }z^{I_-}_{jk}
	&\ge\; & f_{ij}       && \mbox{for all } 1 \le i \le j \le n \label{minB:EinF3} \\
	&  f_{ij}                                         &\in\; & [0,1]    && \mbox{for all } 1 \le i \le j \le n \nonumber \\
        &  z^E_{ik}                                       &\in\; & \{0,1\}    && \mbox{for all } k \in K_E, 1 \le i \le n  \nonumber \\
        &  z^{I_+}_{ik}                                       &\in\; & \{0,1\}    && \mbox{for all } k \in K_I, 1 \le i \le n  \nonumber \\
        &  z^{I_-}_{ik}                                       &\in\; & \{0,1\}    && \mbox{for all } k \in K_I, 1 \le i \le n  \nonumber
\end{align}
The formulation involves binary variables $z^E_{ik}$ to
be equal to $1$ if $i \in B^E_k$ for $k \in K_E$ and equal
to zero otherwise, and
binary variables $z^{I_+}_{ik}$ and $z^{I_-}_{ik}$ to express whether
$i \in B^{I_+}_k$ and $i \in B^{I_-}_k$ for $k \in K_I$. To account for whether
$(i,j) \in Q$, there is a further continuous variable $f_{ij}$ for all
$1 \le i \le j \le n$ that will be equal to $1$ in this case and equal to zero
otherwise.
Constraints \cref{minB:EinF0} fix those $f_{ij}$ to $1$ where the corresponding
pair $(i,j)$ is contained in $P$. Whenever some $j \in N$
is assigned to some set $B_k$, then we induce the corresponding products $(i,j) \in Q$
or $(j,i) \in Q$ for all $i \in A_k$ which is established by inequalities
\cref{minB:EinF1a}--\cref{minB:EinF3b}. Finally, if $(i,j) \in Q$, then we
require Conditions~\ref{cond:c1}--\ref{cond:c3} to be satisfied by inequalities
\cref{minB:EinF1}--\cref{minB:EinF3}, respectively.
The Conditions~\ref{cond:c1}--\ref{cond:c3} impose a certain minimum on the
number of constraints $|B|$
which depends on $P$, the sizes of the sets $A_k$, $k \in K$, and
the distribution of the variables $x_i$, $i \in N$, across them.
In general, different solutions achieving this minimum
may lead to different cardinalities of $|Q|$.
A rational choice for the weights introduced in the objective function
is thus $w_{Q} = 1$ and $w_E = w_{I_+} = w_{I_-} > \max_{k \in K} |A_k|$.
This results in a solution with a minimum number
of constraints that, among these, also induces a minimal number of variables.
Also one might prefer equations by choosing $w_E$ larger compared to the
other weights.
The mixed-integer program is interesting especially for an automated
linearization, e.g.~as part of a mixed-integer programming solver.
It can be significantly simplified if only equations are considered.
Moreover, if in addition the equation comprising each $x_i$, $i \in N$
that is involved in a product is unique, i.e.,
$A_k \cap A_\ell = \emptyset$, for all $k,\ell \in K$, $\ell \neq k$, it
reduces to a linear program with a totally unimodular (TU) constraint
matrix and can alternatively be solved using a simple combinatorial
algorithm as described in~\citet{Mallach2017}. This algorithm
might also be altered such that it is still applicable as a
heuristic in the case of non-disjoint sets $A_k$, $k \in K$.
When considering a particular problem formulation on a paper print, an
associated (most) compact linearization is however
typically \lq recognized\rq\ easily by hand.
\section{Applications}\label{s:Apps}
While the number of existing as well as prospective
applications of the technique proposed is large, we highlight in this
section two prominent combinatorial optimization problems
where formulations found earlier appear now as compact~linearizations.
\subsection{Quadratic Assignment Problem}
Consider a canonical integer programming
formulation for the $n$-by-$n$ quadratic
assignment problem (QAP) in the form
by~\citet{KoopmansBeckmann},
with variables $x_{ip} \in \{0,1\}$ for \lq facilities\rq\ or \lq items\rq\
$i \in \{1, \dots, n\}$ and \lq locations\rq\ or \lq positions\rq\
$p \in \{1, \dots, n\}$. 
Let $y_{ipjq}$ represent the linearization
variable of the product $x_{ip} \cdot x_{jq}$ of any two such variables.
As already mentioned by \citet{Liberti2007}, the following
formulation by~\citet{FriezeYadegar83} can be obtained
by applying the methodology of the compact linearization
technique (and ignoring commutativity in the first place).
\begin{align}
\SwapAboveDisplaySkip
\IPmin & \sum_{i=1}^n \sum_{p=1}^n \sum_{j=1}^n \sum_{q=1}^n d_{ijpq} y_{ipjq} + \sum_{i=1}^n \sum_{p=1}^n c_{ip} x_{ip} \span\span\span\span \nonumber \\
\IPst  & \sum_{i=1}^n x_{ip}      &=\;   & 1       &&\mbox{for all}\; p \in \{1, \dots, n\} \label{qap:assI}\\
       & \sum_{p=1}^n x_{ip}      &=\;   & 1       &&\mbox{for all}\; i \in \{1, \dots, n\} \label{qap:assP}\\
       & \sum_{i = 1}^n y_{ipjq}  &=\;   & x_{jq}  &&\mbox{for all}\; p,j,q \in \{1, \dots, n\} \label{qap:linI}\\
       & \sum_{p = 1}^n y_{ipjq}  &=\;   & x_{jq}  &&\mbox{for all}\; i,j,q \in \{1, \dots, n\} \label{qap:linP}\\
       & \sum_{j = 1}^n y_{ipjq}  &=\;   & x_{ip}  &&\mbox{for all}\; i,p,q \in \{1, \dots, n\} \label{qap:linJ}\\
       & \sum_{q = 1}^n y_{ipjq}  &=\;   & x_{ip}  &&\mbox{for all}\; i,p,j \in \{1, \dots, n\} \label{qap:linQ}\\
       & y_{ipip}                 &=\;   & x_{ip}  &&\mbox{for all}\; i,p \in \{1, \dots, n\}  \label{qap:lin}\\
       & y_{ipjq}                 &\in\; & [0,1]   &&\mbox{for all}\; i,p,j,q \in \{1, \dots, n\} \nonumber\\
       & x_{ip}                   &\in\; & \{0,1\} &&\mbox{for all}\; i,p \in \{1, \dots, n\} \nonumber
\end{align}
For each $y_{ipjq}$, $i,p,j,q \in \{1, \dots, n\}$, the displayed formulation however satisfies
each of the Conditions~\ref{cond:c1} and~\ref{cond:c2} \emph{twice}, i.e., it
is not a compact linearization of minimum size. There is an equivalent formulation
by~\citet{AdamsJohnsonQAP} that comprises only (\ref{qap:linI}) and
(\ref{qap:linP}), and thus satisfies Conditions~\ref{cond:c1} (twice) while
Conditions~\ref{cond:c2} are only \lq indirectly\rq\ satisfied by means of
additional identity constraints $y_{ipjq} = y_{jqip}$ for all
$i,p,j,q \in \{1, \dots, n\}$.
Hence, this formulation cannot, at least not directly, be generated from the
compact linearization approach.
To characterize a \lq most compact\rq\ QAP linearization,
let $K = K_P^E \cup K_I^E$, where $K_P^E$ corresponds to the assignment
constraints (\ref{qap:assI}) and $K_I^E$ corresponds to the assignment
constraints (\ref{qap:assP}). For each $p \in K_P^E$, we have
$A_p = \{ ip \mid i \in \{1, \dots, n\} \}$, and for each $i \in K_I^E$, we have
$A_i = \{ ip \mid p \in \{1, \dots, n\} \}$\footnote{To ease notation, we
treat $ip$ as an index that would of course truly be $i \cdot n + p$.}.
Hence, all the variables $x_{ip}$, $i,p \in \{1, \dots, n\}$,
occur exactly once in $\bigcup_{p \in K_P^E} A_p$ as well as exactly once in
$\bigcup_{i \in K_I^E} A_i$.
Thus, in order to induce all products and to satisfy Conditions~\ref{cond:c1}
and~\ref{cond:c2} for them, it suffices to set
\emph{either} $B_p = \bigcup_{q \in K_P^E} A_q$ for all $p \in K_P^E$ -- which
induces (\ref{qap:linI}) and (\ref{qap:linJ}), \emph{or}
$B_i = \bigcup_{j \in K_I^E} A_j$ for all $i \in K_I^E$ -- which induces
(\ref{qap:linP}) and (\ref{qap:linQ}).
Moreover, since the identities (\ref{qap:lin}) and all variables $y_{ipiq}$ for all pairs
$p,q \in \{1, \dots, n\}$ as well as all variables $y_{ipjp}$ for all pairs $i,j \in \{1, \dots, n\}$ can be
eliminated, it suffices to formulate (\ref{qap:linP}) and (\ref{qap:linQ}) only
for $i \neq j$, and (\ref{qap:linI}) and (\ref{qap:linJ}) only for $p \neq q$.
If one further identifies $y_{jqip}$ with $y_{ipjq}$ whenever $i < j$, it even
suffices to have only exactly \emph{one} of these four equation sets in order to
satisfy Conditions~\ref{cond:c1} and~\ref{cond:c2}. 
The total number of additional equations then reduces to $n^3 - n^2$ compared to
$3 \cdot \big( \frac{1}{2} (n^2 - n) (n^2 -n) \big) = \frac{3}{2} (n^4 - 2n^3 + n^2)$
inequalities when using the Glover-Woolsey linearization and creating
$y_{ipjq}$ only for $i < j$ and $p \neq q$ as well. However, these most
compact formulations have a considerably weaker linear programming relaxation
than the ones by~\citet{FriezeYadegar83} and~\citet{AdamsJohnsonQAP}.
\subsection{Symmetric Quadratic Traveling Salesman Problem}\label{ss:QTSP}
The symmetric quadratic traveling salesman problem asks for a
tour $T \subseteq E$ in a complete undirected graph $G=(V,E)$ such that
the objective
$\sum_{\{i,j,k\} \subseteq V, j \neq i < k \neq j } c_{ijk} x_{ij} x_{jk}$
(where $x_{ij} = 1$ if and only if $\{i,j\} \in T$) is minimized.
Consider the following mixed-integer programming formulation for this
problem as presented by~\citet{FischerH13} and oriented at the integer
programming formulation for the linear traveling
salesman problem by~\citet{Dantzig54}.
\begin{align}
\IPmin & \sum_{\{i,j,k\} \subseteq V, j \neq i < k \neq j } c_{ijk} y_{ijk} \span\span\span\span \nonumber \\
\IPst  & \sum_{\{i,j\} \in E} x_{ij}  &=\;   & 2     &&\mbox{for all}\; i \in V \label{qtsp:degree}\\
       & x(E(W))                                 &\le\; & |W|-1  &&\mbox{for all}\; W \subsetneq V,\; 2 \le |W| \le |V| - 2 \nonumber \\ 
       & y_{ijk}                                 &=\; & x_{ij} x_{jk} &&\mbox{for all}\; \{i,j,k\} \subseteq V, j \neq i < k \neq j \label{qtsp:lin}\\
       & x_{ij}                                 &\in\; & \{0,1\}     &&\mbox{for all}\; \{i,j\} \in E \nonumber
\end{align}
In the context of the approach presented, we
consider only the linear equations~(\ref{qtsp:degree}) so that we have
$K = K_E = V$, $A_k = \{ jk \mid j < k \mbox { and } \{j,k\} \in E \}$,
$\alpha_i^k = 1$ for all $i \in A_k$ and $\beta^k = 2$ for all $k \in K$.
Since we are interested in the bilinear terms of the form as in~(\ref{qtsp:lin}),
i.e.~in each pair of edges with common index $j$, we need to set $B_k = A_k$ for
all $k \in K$ in order to satisfy both Conditions~\ref{cond:c1} and~\ref{cond:c2}
for each such pair.
We thus comply to the requirements of the special case addressed
in Theorem~\ref{thm:tsp} (Sect.~\ref{ss:ProveStrong}) and obtain the equations:
\begin{align}
    \sum_{\{i,j\} \in E}  x_{ij} x_{jk}     &=\; 2 x_{jk} && \mbox{for all } \{j,k\} \in E,\; \mbox{for all } j \in V \nonumber 
\end{align}
After introducing linearization variables with indices ordered as desired, these are resolved as:
\begin{align}
    \sum_{\{i,j,k\} \subseteq V, j \neq i \le k \neq j} y_{ijk} &=\; 2 x_{jk} && \mbox{for all } \{j,k\} \in E,\; \mbox{for all } j \in V \nonumber 
\end{align}
Each of these equations induces one variable more than originally demanded,
namely $y_{kjk}$ as the linearized substitute for the square term $x_{jk} x_{jk}$.
Thus we may safely subtract $y_{kjk}$ from the left and $x_{jk}$ from the right
hand side and obtain
\begin{align}
    \sum_{\{i,j,k\} \subseteq V, j \neq i < k \neq j} y_{ijk} &=\; x_{jk} && \mbox{for all } \{j,k\} \in E,\; \mbox{for all } j \in V \nonumber
\end{align}
which are exactly the linearization constraints as presented by~\citet{FischerH13}.
\section{Conclusion}\label{s:Concl}
As has been shown in this paper, the compact linearization technique
can be applied not only to binary quadratic problems with assignment
constraints, but to those with arbitrary linear constraints with
positive coefficients and right hand sides. We discussed two particular cases where
the continuous relaxation of the obtained compactly linearized
problem formulation is provably as least as strong as the one
obtained with the well-known linearization by~\citet{GloverWoolsey74}.
Moreover, we highlighted
previously found formulations for the quadratic assignment problem
and the symmetric quadratic traveling salesman problem that appear
as special cases that result when applying the proposed method.
Last but not least, we demonstrated how a compact linearization
can be generated automatically.
\section*{Acknowledgments}
I would like to thank Anja Fischer for bringing up the question
whether the original approach for assignment constraints might
be possibly generalized to arbitrary right hand sides and
pointing me to her research about the symmetric quadratic traveling
salesman problem. Just as well I want to thank Sourour Elloumi for
encouraging me to consider the inequality cases in more detail.
\bibliographystyle{spbasic}
\bibliography{bibfile}
\end{document}